\documentclass[12pt,oneside]{article}

\usepackage{amssymb}

\usepackage{amsmath}

\usepackage{amsfonts}

\usepackage{longtable}

\newtheorem{theorem}{Theorem}[section]
\newtheorem{lemma}[theorem]{Lemma}

\newtheorem{proposition}[theorem]{Proposition}

\newtheorem{corollary}[theorem]{Corollary}

\newtheorem{definition}[theorem]{Definition}

\newenvironment{ack}
{\begin{trivlist}  \item \textsc{Acknowledgments}~} {\end{trivlist}}

\newenvironment{proof}
{\begin{trivlist}  \item \textsc{Proof:}~} {\hfill $\Box$
\end{trivlist}}

\newenvironment{proof of claim}
{\begin{trivlist}  \item \textsc{Proof of Claim:}~} {\hfill $\Box$
\end{trivlist}}

\newcommand{\closure}[1]{\ensuremath{\mathrm{cl}}(#1)}
\newcommand{\interior}[1]{\ensuremath{\mathrm{int}}(#1)}

\def \st {\operatorname{st}}

\def \R{\mathbb{R}}
\def \Q{\mathbb{Q}}
\def \N{\mathbb{N}}
\def \ind{\operatorname{ind}}
\def \supp {\operatorname{supp}}
\newcommand{\stinv}[1]{\ensuremath{\mathrm{st}^{-1}}(#1)}

\newcommand{\ma}{\mathfrak{m}}

\begin{document}
\author{
Jana Ma\v{r}\'{i}kov\'{a}}
\title{The structure on the real field generated by the standard part map on an
o-minimal expansion of a real closed field}
\maketitle

\begin{abstract}
Let $R$ be a sufficiently saturated o-minimal expansion of a real
closed field, let $\mathcal{O}$ be the convex hull of $\mathbb{Q}$
in $R$, and let $\text{st} : \mathcal{O}^n \rightarrow \mathbb{R}^n$
be the standard part map. For $X \subseteq R^n$ define
$\st{X}:=\st{(X \cap \mathcal{O}^n )}$. We let $\mathbb{R}_{\ind}$
be the structure with underlying set $\mathbb{R}$ and expanded by
all sets of the form $\st{X}$, where $X \subseteq R^{n}$ is
definable in $R$ and $n=1,2,\dots$. We show that the subsets of
$\mathbb{R}^n$ that are definable in $\mathbb{R}_{\ind}$ are exactly
the finite unions of sets of the form $\st{X}\setminus \st{Y}$,
where $X,Y \subseteq R^n$ are definable in $R$. A consequence of
the proof is a partial answer to a question by Hrushovski,
Peterzil and Pillay about the existence of measures with certain
invariance properties on the lattice of bounded
definable sets in $R^n$.
\end{abstract}

\begin{section}{Introduction}

\noindent Throughout, $\mathbb{N}=\{0,1,2,\dots\}$ and $m,n$ range
over $\mathbb{N}$.

\medskip\noindent
Let $R$ be an o-minimal expansion of an ordered field (necessarily real closed), let $\mathcal{O}=\{a\in R:\ |a| \le n \text{ for some }n\}$
be the convex hull of $\mathbb{Q}\subseteq R$ in $R$,
and let $\ma$ be the maximal ideal of the valuation ring
$\mathcal{O}$ of $R$, so $\ma=\{a\in R:\ |a| \le 1/n \text{ for all }n>0\}$.
Let $\st: \mathcal{O} \to \mathbb{R}$ be the standard part map; it has
kernel $\ma$ and induces for each $n$ a corresponding
standard part map $\st: \mathcal{O}^n \to \mathbb{R}^n$. For $X\subseteq R^n$
we set $\st(X):= \st(X\cap \mathcal{O}^n)$.

From now on we assume that $R$ is $(2^{\aleph_0})^+$-saturated. In particular, the map
$\st: \mathcal{O} \to \R$ is surjective, and if
$X\subseteq R^3$ is the graph of the addition operation of $R$, then $\st(X)\subseteq \R^3$ is the graph of the addition operation of $\R$.  The same is true for multiplication instead of addition.

By {\em definable\/} we shall mean
{\em definable with parameters in the structure $R$}, unless specified
otherwise.
If another ambient structure is specified, then {\em definable\/} also means {\em definable with parameters\/} (in that structure).

Via the standard part map the definable sets of $R$ induce a
structure on $\R$ as follows: let $\R_{\ind}$ be the structure
with underlying set $\R$ and with
the sets $\st(X)$ with definable $X\subseteq R^n$, $n=0,1,2,\dots$,
 as basic
relations. Since the graphs of the addition and multiplication
 on $\R$
are among these basic relations, and the usual ordering of $\R$
 is $0$-definable
from addition and multiplication, we may view $\R_{\ind}$ as an
expansion of the ordered field of real numbers, and we shall do
so. It follows from a theorem by Baisalov and Poizat \cite{bp}
that $\R_{\ind}$ is o-minimal;
 this was observed
by Hrushovski, Peterzil and Pillay \cite{nip}, but their argument
left open how logically complicated the definable relations of
$\R_{\ind}$ can be, compared to the basic relations. We answer this
question here as follows:

\begin{theorem}\label{intrthm} The subsets of $\R^n$ definable in $\R_{\ind}$
are exactly the
finite unions of differences $\st(X) \setminus \st(Y)$ with
definable $X,Y\subseteq R^n$.
\end{theorem}

\noindent
This result is obtained without using the
Baisalov-Poizat theorem, and thus gives another proof of the fact
that $\R_{\ind}$ is o-minimal. A previously known special case of
Theorem~\ref{intrthm} is when $R$ is an elementary extension of an o-minimal
expansion $\R^{\#}$ of the ordered field of real numbers; see
\cite{tconvex}. (The key fact in that case is that $\R_{\ind}$ and
$\R^{\#}$ have the same definable relations.)

\medskip\noindent
The proof of the theorem goes as follows. We single out
certain subsets of $\R^n$ as {\em good cells}; they have the form
$\st(X) \setminus \st(Y)$ with definable $X,Y\subseteq R^n$, and for
$n>0$ the image of a good cell in $\R^n$ under the projection map
$(x_1,\dots, x_n)\mapsto (x_1,\dots,x_{n-1})$ is a good cell in
$\R^{n-1}$. The main step is to show by induction on $n$ that for
any definable $X\subseteq R^n$ the set $\st(X)$ is a finite union of
good cells. More precisely, we have ``good cell decomposition'',
Corollary~\ref{reduction}. The theorem above then follows easily.

We also show that the closed subsets of $\R^n$ definable
in $\R_{\ind}$ are exactly the sets $\st(X)$ with definable
$X\subseteq R^n$.

As a consequence of a strengthening of good cell decomposition we
obtain a partial answer to a question posed in \cite{nip}, which is
roughly as follows. Let $B[n]$ be the lattice of all
bounded definable subsets of $R^n$, and define $X, Y \in B[n]$ to be
{\em isomorphic\/} iff, modulo a set of dimension $<n$, we have
$\psi (X) = Y$ for some definable $C^1$-diffeomorphism $\psi$ with
$|J \psi (x)|=1$ for all $x\in X$. Let $X \in B[n]$ have nonempty
interior. Is there a finitely additive $\mu : B[n]\rightarrow
[0,\infty]$ with $0< \mu (X) < \infty$ which is invariant under
isomorphisms?

Our partial result is that there is such a $\mu$ if
$X\subseteq \mathcal{O}^n$ and $\st(X)$ has nonempty interior.
This follows by proving that the measure introduced by Berarducci
and Otero in \cite{bo} on the lattice of definable sets contained in
$\mathcal{O}^n$ is invariant under isomorphism.
The main point here is that the standard part of a partial derivative of
a definable function is almost everywhere equal to the corresponding
partial derivative of the standard part of the function.

\bigskip\noindent
{\bf Further notations and terminology.}
An {\em interval} is always a {\em nonempty open interval}
$(a,b)$,
and intervals are in
$R$ or in $\mathbb{R}$, as specified. For $m\le n$ we let
$p^n_m: R^n \to R^m$ and
$\pi^n_m: \mathbb{R}^n \to \mathbb{R}^m$ be given by
$$p^n_m(x_1,\dots,x_n)=(x_1,\dots,x_m), \qquad
\pi^n_m(x_1,\dots,x_n)=(x_1,\dots,x_m).$$ The {\em hull\/} of a set
$C\subseteq \mathbb{R}^n$ is by definition the clopen set $C^h:=
\stinv{C}\subseteq \mathcal{O}^n$. A point $x$ in $R^n$ or $\R^n$
has components $x_1,\dots,x_n$, that is, $x=(x_1,\dots,x_n)$.

Let $f : X \rightarrow R$, where $X\subseteq R^n$.
Then the graph of $f$ as a subset of $R^{n+1}$ is denoted by $\Gamma{f}$, and
we put
\begin{align*} (-\infty, f)&:=\{(x,y)\in R^{n+1} : x \in X \, \& \,
y < f(x) \},\\
     (-\infty, f]&:=\{(x,y)\in R^{n+1} : x \in X \, \& \,
y \le f(x) \},\\
     (f,+\infty)&:= \{(x,y)\in R^{n+1} : x \in X \, \& \,  f(x)<y \},\\
     [f,+\infty)&:= \{(x,y)\in R^{n+1} : x \in X \, \& \,  f(x) \le y \}.
\end{align*}
When also $g : X \rightarrow R$, then
``$f<g$'' abbreviates ``$f(x) < g(x)$ for all $x \in X$'' and if $f<g$ we put
$$(f,g) := \{
(x,y)\in R^{n+1} : x \in X \, \& \, f(x) < y < g(x) \}.$$
Likewise, functions $X\to \mathbb{R}$ with
$X\subseteq \mathbb{R}^n$ give rise to subsets of $\mathbb{R}^{n+1}$ that we
denote in the same way.
A $\mathbb{Q}$-box in $R^n$ is a cartesian product
$$I_1 \times \dots \times I_n\subseteq R^n$$ of intervals $I_j$ in $R$
whose endpoints lie in $\Q$. Any unexplained terminology or notation
is from \cite{lou}.

\section{Basic facts about standard part sets}

\noindent
It is easy to see that if
$X\subseteq R^n$ is definable in $R$, then $\st{X}$ is closed in
$\mathbb{R}^n$. Let $\text{St}_n$ be the collection
of all sets $\st{X}$ with definable $X\subseteq R^n$.

Note: if $X,Y\in \text{St}_n$, then $X\cup Y\in \text{St}_n$; if
$X\in \text{St}_m$ and $Y\in \text{St}_n$, then $X\times Y\in
\text{St}_{m+n}$.
The next lemma is almost obvious, with (1) a special case of (2). To state it
we use the projection maps $\pi=\pi^{m+n}_m: \mathbb{R}^{m+n}\to \mathbb{R}^m$ and $p=p^{m+n}_m: R^{m+n} \to R^m$.

\begin{lemma}\label{stan0} Let $X\in \operatorname{St}_{m+n}$. Then
\begin{enumerate}
\item[$(1)$] if $X$ is bounded, then $\pi(X)\in \operatorname{St}_m$,
\item[$(2)$] if $X=\st{X'}$ where the set $X'\subseteq R^{m+n}$ is
definable in $R$ and satisfies $X'\cap
p^{-1}(\mathcal{O}^m)\subseteq \mathcal{O}^{m+n}$, then $\pi(X)\in
\operatorname{St}_m$.
\end{enumerate}
\end{lemma}

\begin{lemma}\label{stan1} If $X,Y\in \operatorname{St}_n$, then
$X\cap Y\in \operatorname{St}_n$.
\end{lemma}
\begin{proof} Let $X,Y\in \text{St}_n$, and take definable
$X', Y'\subseteq R^n$ such that $\st(X')=X$ and $\st(Y')=Y$. For
each $a\in X\cap Y$ take $x_a\in X'$ and $y_a\in Y'$ such that
$\st(x_a)=\st(y_a)=a$. By saturation (in a cardinal $>2^{\aleph_0}$)
we can take an infinitesimal $\varepsilon\in R^{>0}$ such that
$d(x_a,y_a)< \varepsilon$ for all $a\in X \cap Y$. Hence, with
$$Z:= \{(x,y)\in X'\times Y': d(x,y) < \varepsilon\}\subseteq R^{2n},$$
$Z$ is definable and $X\cap Y$ is the image of
$\st(Z)\subseteq \mathbb{R}^{2n}$ under the projection map
$\pi^{2n}_n: \mathbb{R}^{2n} \to
\mathbb{R}^n$. Now apply (2) of Lemma~\ref{stan0}.
\end{proof}

\begin{lemma}\label{stan1a} Let $X\subseteq R^n$ and $f: X \to R$ be
definable, and put
$$  X^{-}:=\{x\in X:\ f(x) < \mathbb{Q}\}, \quad
X^{+}:=\{x\in X:\ f(x)> \mathbb{Q}\}.$$ Their standard parts $\st
(X^{-})$ and $\st(X^{+})$ belong to $\operatorname{St}_n$.
\end{lemma}
\begin{proof} To get $\st(X^{-})\in \text{St}_n$, use Lemma~\ref{stan1},
the fact that
$$Y:= \{(x,y)\in X\times R:\ f(x)<0,\  f(x)\cdot y=1\}\subseteq R^{n+1}$$
is definable, and
$$\st(X^{-})\ =\ \pi^{n+1}_n\big((\st{Y}) \cap
(\mathbb{R}^n \times \{0\})\big).$$ In the same way we see that
$\st(X^{+})\in \text{St}_n$.
\end{proof}

\begin{lemma}\label{stan2} If $X\subseteq R$ is definable, then
$\operatorname{st}(X)$
is a finite union of intervals and points in $\mathbb{R}$.
\end{lemma}
\begin{proof} This is immediate from the o-minimality of $R$.
\end{proof}

\section{Good cells}

\noindent
The following notion turns out to be very useful.

\begin{definition}
Given functions $f: X \rightarrow R$ with $X \subseteq R^n$, and $g:C
\rightarrow \R$ with $C \subseteq \R^n$, we say that $f$ {\em
induces\/} $g$ if $f$ is definable $($so $X$ is definable$)$,
$C^h\subseteq X$, $f|C^h$ is continuous, $f(C^{h}) \subseteq \mathcal{O}$ and
$\Gamma g = \st(\Gamma f) \cap (C\times \mathbb{R})$.
\end{definition}


\begin{lemma} Let $C \subseteq \mathbb{R}^{n}$ and suppose $g: C \rightarrow \mathbb{R}$ is induced by the function $f: X \rightarrow R$ with
$X\subseteq R^n$. Then $g$ is continuous.
\end{lemma}

\begin{proof}
Let $x \in C$ and suppose towards a contradiction that
$\epsilon \in \mathbb{Q}^{>0}$ is such that for every
$\lambda \in \mathbb{Q}^{>0}$ we have $x_{\lambda} \in C$ with
$|x_{\lambda} - x|< \lambda$ and $|g(x_{\lambda}) - g(x)| > \epsilon$.
Pick $y \in \{x\}^{h}$ and for $\lambda \in \mathbb{Q}^{>0}$ pick
$y_{\lambda} \in \{ x_{\lambda} \}^{h}$. Then
$|f(y) - f(y_\lambda )| \ge \epsilon$ for those $\lambda$, so by
saturation we get a point
 $z \in  \{x\}^{h}$ with $|f(y) - f(z )| \ge \epsilon$,
contradicting that $g$ is a function.
\end{proof}

\noindent
For $C\subseteq \mathbb{R}^n$ we let $G(C)$ be the set of all $g:
C \to \mathbb{R}$ that are induced by some definable $f: X \to R$ with
$X\subseteq R^n$.

\begin{lemma}\label{stan4}
Let $1\le j(1) < \dots < j(m)\le n$ and define
$$\pi: \mathbb{R}^n \to \mathbb{R}^m, \quad  \pi(x_1,\dots,x_n)=(x_{j(1)},
\dots,x_{j(m)}).$$ Let $C\subseteq \mathbb{R}^n$ and suppose
$g\in G(\pi C)$. Then
$g\circ \pi|_{C} \in G(C)$.
\end{lemma}
\begin{proof} Take definable $f: Y \to R$ with $Y\subseteq R^m$
such that $f$ induces $g$, so $\Gamma g = \st(\Gamma f) \cap
\big(\pi C\times \mathbb{R}\big)$.  Let $p: R^n \to R^m$ be given by
$$p(x_1,\dots,x_n)=(x_{j(1)},\dots,x_{j(m)}),$$ and put
$X:= p^{-1}(Y)$. Then $C^h \subseteq X$, and it is easy to check
that
$$\Gamma (g\circ \pi|_{C})=\st\big(\Gamma(f \circ p|_{X})\big)\cap (C\times
\R),$$ so $g\circ \pi|_{C}$ is induced by $f\circ p|_{X}$.
\end{proof}


\begin{definition}\label{defgood} Let $i=(i_1,\dots,i_n)$ be a sequence of zeros and ones. Good $i$-cells are subsets of $\R^n$ obtained by recursion on
$n$ as follows:
\begin{enumerate}
\item[{\em (i)}] For $n=0$ and $i$ the empty sequence,
the set $\R^0$ is the only good $i$-cell, and
for $n=1$, a {\em good $(0)$-cell\/} is a singleton
$\{a\}$ with $a\in \R$;
a {\em good $(1)$-cell\/} is an interval in $\R$.

\item[{\em (ii)}] Let $n>0$ and assume inductively that good $i$-cells are
subsets of $\R^n$. A good $(i,0)$-cell is a set $\Gamma h \subseteq
\mathbb{R}^{n+1}$ where $h \in G(C)$ and $C \subseteq
\mathbb{R}^n$ is a good $i$-cell. A good $(i,1)$-cell is either a
set $C\times \R$, or a set
$(-\infty, f)\subseteq \R^{n+1}$, or a set $(g,h) \subseteq
\mathbb{R}^{n+1}$, or a set $(f, +\infty) \subseteq
\mathbb{R}^{n+1}$, where $f,g,h \in G(C)$, $g<h$, and $C$ is a good
$i$-cell.
\end{enumerate}
\end{definition}

\noindent
One verifies easily that
a good $i$-cell is open in $\mathbb{R}^n$
iff $i_1=\dots =i_n=1$, and that if $i_1=\dots=i_n=1$, then every
good $i$-cell is homeomorphic to $\mathbb{R}^n$.
A {\em good cell in
$\mathbb{R}^n$} is a good $i$-cell for some sequence
$i=(i_1,\dots,i_n)$ of zeros and ones.

\begin{lemma}\label{gooddifference}
Every good cell in $\R^n$ is of the form
$X \setminus Y$
with $X,Y\in \operatorname{St}_n$.
\end{lemma}
\begin{proof} This is clear for $n=1$. Suppose it holds for a certain
$n\ge 1$, and
consider first an $(i,0)$-cell $\Gamma h\subseteq \R^{n+1}$ as
in (ii) above, with $h\in G(C)$ induced by $f: X \to R$, where
$X\subseteq R^n$.
Then
$\Gamma h=\st( \Gamma f) \cap (C\times \R)$.
 Now $C=\st(P) \setminus \st(Q)$ with definable $P,Q\subseteq R^n$, so
 $C\times \R= \st(P\times R) \setminus \st(Q\times R)$, hence
$$\Gamma h = \big(\st(\Gamma f) \cap \st(P\times R)\big)\setminus \big(\st(\Gamma f) \cap \st(Q\times R)\big),$$
and we are done by Lemma~\ref{stan1}.  Next, consider an
$(i,1)$-cell $(g,h)\subseteq \R^{n+1}$
with $g, h\in G(C)$, $g<h$, with $g$ induced by  $\phi: X \to R$ and
$h$ induced by $\psi: Y \to R$ with $X,Y\subseteq R^n$. Then
$\Gamma g=\st( \Gamma \phi) \cap (C\times \R)$ and
$\Gamma h = \st(\Gamma \psi)\cap (C\times \R)$.
It is easy to check that
\begin{align*}
(-\infty, g]&= \st((-\infty, \phi]) \cap (C\times \R),\\
[h, +\infty)&=\st([\psi,+\infty)) \cap (C\times \R), \text{ hence}\\
(g,h)&=(C\times \R) \setminus \st\big((-\infty, \phi] \cup [\psi, +\infty)\big).
\end{align*}
Now $C=\st(P) \setminus \st(Q)$ with definable $P,Q\subseteq R^n$, so
$$(g, h)= \st(P\times R) \setminus \st\big( (Q\times R) \cup (-\infty, \phi] \cup [\psi,+\infty)\big),$$
and we are done. The other types of $(i,1)$-cells are
treated likewise.
\end{proof}

\begin{lemma}\label{inverse}
Let $C\subseteq \R^n$ be a good $(i_1,\dots,i_n)$-cell, and suppose $i_k=0$
where $k\in \{1,\dots,n\}$. Let $\pi: \R^n \to \R^{n-1}$ be given by
$$ \pi(x_1,\dots,x_n)=(x_1,\dots,x_{k-1}, x_{k+1}, \dots,x_n).$$
Then $\pi(C)\subseteq \R^{n-1}$ is a good cell,
$\pi|C: C \to \pi(C)$ is a homeomorphism, and if
$E\subseteq \pi(C)$ is a good cell, so is its inverse image
$\pi^{-1}(E)\cap C$.
\end{lemma}
\begin{proof} By induction on $n$. If $n=1$, then $k=1$ and $C$
is a singleton, and the lemma holds trivially in that case.
Assume inductively that the lemma holds for a certain $n>0$,
let $C\subseteq \R^{n+1}$
be a good $(i_1,\dots,i_{n+1})$-cell, let
$k\in \{1,\dots,n+1\}$ be such that $i_k=0$, and let
$\pi: \R^{n+1} \to \R^{n}$
be given by
$$ \pi(x_1,\dots,x_{n+1})=(x_1,\dots,x_{k-1}, x_{k+1}, \dots,x_{n+1}).$$
Our task is to establish the following.

\medskip\noindent
{\em Claim}.\  $\pi(C)\subseteq \R^{n}$ is a good cell,
$\pi|C:\ C \to \pi(C)$ is a homeomorphism, and if
$E\subseteq \pi(C)$ is a good cell, then
$\pi^{-1}(E)\cap C$ is a good cell in $\R^{n+1}$.

\medskip\noindent
If $k=n+1$, then $\pi=\pi_n^{n+1}$ and $C=\Gamma f$ with $f\in G(\pi(C))$, and
then the claim follows easily.
So we can assume $k\le n$. Then we introduce the good cell
$D:= \pi_n^{n+1}(C)$ in $\R^n$ and the map $\pi_0: \R^n \to \R^{n-1}$
defined by
$$ \pi_0(x_1,\dots,x_n)=(x_1,\dots,x_{k-1}, x_{k+1}, \dots,x_n).$$
Then $\pi_0(D)\subseteq \R^{n-1}$ is a good cell,
$\pi_0|_D: D \to \pi_0(D)$ is a homeomorphism, and for each good cell
$F\subseteq \pi_0(D)$ its inverse image
$\pi_0^{-1}(F)\cap D$ is a good cell in $\R^n$.
Since $\pi(x,t)=(\pi_0(x),t)$ for $x\in \R^n$ and $t\in \R$, it follows that
$\pi|_{D\times \R}: D\times \R \to \pi_0(D)\times \R$ is a homeomorphism, so
$\pi|C:\ C \to \pi(C)$ is a homeomorphism. We have
$\pi_k^n D=\Gamma h$ where $h\in G\big(\pi_{k-1}^n(D)\big)$, and the
map $(\pi_0|_D)^{-1}: \pi_0(D) \to D$ is given by
$$(x_1,\dots,x_{k-1},x_{k+1},\dots,x_n) \mapsto
(x_1,\dots,x_{k-1},h(x_1,\dots,x_{k-1}), x_{k+1},\dots,x_n).$$
Let $h$ be induced by $\eta: Y \to R$, $Y\subseteq R^{k-1}$.

Consider first the case that $C=\Gamma f$ with $f\in G(D)$.
It is routine to check that then $\pi(C)=\Gamma{f_0}$, where
$f_0: =f\circ (\pi_0|_D)^{-1}\ :\ \pi_0(D) \to \R$ is given by
$$(x_1,\dots,x_{k-1},x_{k+1},\dots,x_n)\mapsto
f(x_1,\dots,x_{k-1},h(x_1,\dots,x_{k-1}), x_{k+1},\dots,x_n).
$$
Let $f$ be induced by $\phi: X \to R$, $X\subseteq R^n$, and let
$Z$ be the set of all $(x_1,\dots,x_{k-1},x_{k+1},\dots,x_n)\in R^{n-1}$ such
that
$$(x_1,\dots,x_{k-1})\in Y,\ (x_1,\dots,x_{k-1},\eta(x_1,\dots,x_{k-1}),x_{k+1},\dots,x_n)\in X.$$
One easily shows that then $f_0$ is induced by the function $Z \to R$ given by
$$(x_1,\dots,x_{k-1},x_{k+1},\dots,x_n) \mapsto
\phi(x_1,\dots,x_{k-1},\eta(x_1,\dots,x_{k-1}),x_{k+1},\dots,x_n).$$
Thus $\pi(C)=\Gamma{f_0}$ is a good cell in $\R^n$. Let $E\subseteq \pi(C)$
be a good cell. Then $E=\Gamma(f_0|_F)$ where $F\subseteq \pi_0(D)$
is a good cell, so $B:= \pi_0^{-1}(F)\cap D$ is a good cell in $\R^n$ by the
inductive assumption. Then $\pi^{-1}(E)\cap C = \Gamma(f|_B)$, as is easy to
check, so $\pi^{-1}(E)\cap C$ is indeed a good cell.

Next, consider the case $C=(f,g)$ where $f, g\in G(D)$, $f < g$.
Then $\pi(C)=(f_0, g_0)$, where
\begin{align*}f_0:&=f\circ (\pi_0|_D)^{-1}\ :\ \pi_0(D) \to \R,\\
g_0:&=g\circ (\pi_0|_D)^{-1}\ :\ \pi_0(D) \to \R,
\end{align*} and as before one
checks that $f_0, g_0\in G(\pi_0(D))$, so $\pi(C)$ is a good cell.
Let $E\subseteq \pi(C)$ be a good cell, and set $F:= \pi^n_{n-1}(E)$. Then
$F\subseteq  \pi_0(D)$ is a good cell, so $B:= \pi_0^{-1}(F)\cap D$ is a
good cell in $\R^n$ by the inductive assumption. If
$E=\Gamma s$ with $s\in G(F)$, then
$$\pi^{-1}(E)\cap C = \Gamma(s\circ \pi_0|_B),$$ as is easy to
check, and  $(s\circ \pi_0|_B) \in G(B)$ by Lemma~\ref{stan4}, so
$\pi^{-1}(E)\cap C$ is indeed a good cell. If $E=(s,t)$ with
$s,t\in G(F)$, $s<t$, then
$$\pi^{-1}(E)\cap C = (s\circ \pi_0|_B,t\circ \pi_0|_B),$$ and
$(s\circ \pi_0|_B), (t\circ \pi_0|_B) \in G(B)$ by Lemma~\ref{stan4}, so
$\pi^{-1}(E)\cap C$ is indeed a good cell.

The remaining cases, where $C=D\times \R$, or $C=(-\infty,f)$, or
$C=(f,+\infty)$, with
$f\in G(D)$, are treated in the same way.
\end{proof}

\section{Good cell decomposition}

\noindent A set $X\subseteq R^n$ is said to be {\em strongly
bounded\/} if there is $q\in \Q^{>0}$ such that $|x|\le q$ for all $x\in X$.
The proof of good cell
decomposition in this section works initially only for strongly bounded
definable sets, because it uses part (1) of Lemma~\ref{stan0}.
Once we have good cell decomposition for that case we
extend it to general definable sets using the homeomorphism
$x \mapsto x/\sqrt{1+x^2}:\ \R\ \to\ (-1,1)$.

Berarducci and Otero~\cite{bo} define
a real-valued finitely additive measure $\mu= \mu^{(n)}$ on the
lattice of strongly bounded definable subsets of $R^n$. The
properties of this measure imply a fact that is useful for the
inductive step in the proof of good cell decomposition:

\begin{lemma}\label{Qdensitylemma}
Syppose $Y \subseteq R^n$ is definable and $\st{Y}$ has nonempty interior in
$\R^n$. Then $Y$ contains a $\mathbb{Q}$-box.
\end{lemma}
\begin{proof} We can assume $Y$ is strongly bounded. Then by Theorem 4.3 of
\cite{bo} we have $\mu (Y) = \lambda (\st{Y} )$ where $\lambda$ is the
usual Lebesgue measure on $\R^n$; in particular, $\mu(Y)>0$. The way
$\mu$ is defined in \cite{bo} guarantees that $Y$ contains a
$\mathbb{Q}$-box.
\end{proof}

\begin{lemma}\label{proj}
Let $C \subseteq \mathbb{R}^{n}$ be a good $i$-cell, let $X
\subseteq R^{n+1}$ be definable and suppose $k\in \{1,\dots,n\}$ is
such that $i_k=0$. Define $\pi : \mathbb{R}^{n+1} \rightarrow
\mathbb{R}^{n}$ by
$$\pi(x)= (x_1 , \dots , x_{k-1} , x_{k+1},
\dots ,x_{n+1} ).$$
Then $\pi \big(\st(X) \cap (C \times \mathbb{R})\big)$
is a difference of sets in $\operatorname{St}_n$.
\end{lemma}

\begin{proof}
Let $\pi_{k}^{n} C = \Gamma g$, with $g: \pi^{n}_{k-1}C \rightarrow
\mathbb{R}$ induced by $f : Y \rightarrow R$, $Y\subseteq R^{k-1}$.
For infinitesimal
$\varepsilon \in R^{>0}$, define $X_{\varepsilon}
\subseteq X$ as follows:
\[X_{\varepsilon}\ :=\ \{x \in X: \; (x_1 ,
\dots ,x_{k-1}) \in Y \mbox{ and }|f(x_1 , \dots ,x_{k-1}) -
x_{k}| \le \varepsilon \}
\]

\medskip\noindent
{\em Claim 1. \/} There is an infinitesimal $\varepsilon \in R^{>0}$
such that
$$\st(X) \cap (C \times \mathbb{R}) = \st(X_{\epsilon}) \cap (C
\times \mathbb{R}).$$

\medskip\noindent
To see this, pick for each $a \in \st(X) \cap (C \times \mathbb{R})$, an
$x \in \stinv{a}$. For such $x$,
$$\st (x_1 , \dots ,x_{k-1 }) \in \pi_{k-1}^{n}C\ \text{ and }\
|f(x_1 , \dots ,x_{k-1}) - x_{k}| \text{ is infinitesimal}.$$
Then saturation gives an infinitesimal $\varepsilon \in R^{>0}$ as
claimed.

\medskip\noindent Define $p: R^{n+1} \rightarrow R^{n}$ by
$p(x)=(x_1 \dots , x_{k-1}, x_{k+1} , \dots , x_{n+1})$, and take an
infinitesimal $\varepsilon\in R^{>0}$ with the property of Claim 1.

\medskip\noindent
{\em Claim 2.\/}\ $\pi \big(\st(X_{\varepsilon}) \cap (C \times
\mathbb{R})) = \st{p(X_{\varepsilon})} \cap \pi(C \times
\mathbb{R})$.

\medskip\noindent
It is clear that $\pi \big(\st(X_{\varepsilon}) \cap (C
\times \mathbb{R})\big) \subseteq \st{p(X_{\varepsilon})} \cap
\pi(C \times \mathbb{R})$. So take $x \in X_{\varepsilon}$ such
that
 $(\st{x_1} , \dots ,\st{x_{k-1}}, \st{x_{k+1}},
 \dots ,\st{x_{n+1}}) \in
\pi(C \times \R )$. We claim that then
$$\st x \in \st(X_{\varepsilon}) \cap (C
\times \mathbb{R}).$$
This follows from the definition of
$X_{\varepsilon}$: clearly $\st(x_1 , \dots x_{k-1}) \in
\pi_{k-1}^{n}C$ and $|x_{k} - f(x_1 , \dots ,x_{k-1})|$ is
infinitesimal. Hence $$\st{x_{k}} = \st{f (x_1 , \dots ,x_{k-1})}
= g (\st(x_1 , \dots ,x_{k-1})),$$ and so $\st(x_1 , \dots
,x_{k}) \in \Gamma{g}$.
\end{proof}

\noindent
We set $I:=[-1,1] \subseteq \mathbb{R}$ and $I(R):=\{x \in R: -1
\leq x \leq 1 \}$. A {\em good decomposition of $I^n$\/} is a
special kind of partition of $I^n$ into finitely many good cells.
The definition is by recursion on $n$:
\begin{enumerate}
\item[(i)] a good decomposition of $I^1 = I$ is a collection $$\{(c_0 ,
c_1 ),(c_2 , c_3 ), \dots ,(c_k , c_{k+1}), \{c_0 \}, \{c_1 \} ,
\dots ,\{c_k \},\{c_{k+1} \} \}$$ of intervals and points in $\R$
where $c_0 < c_1 < \dots < c_k < c_{k+1}$ are real numbers with
$c_0 = -1$ and $c_{k+1}=1$;
\item[(ii)] a good decomposition of
$I^{n+1}$ is a finite partition $\mathcal{D}$ of $I^{n+1}$ into good
cells such that $\{\pi_{n}^{n+1}C:\; C \in \mathcal{D}  \}$ is a
good decomposition of $I^n$.
\end{enumerate}

\noindent

\begin{theorem}\label{thm} {\bf (Good Cell Decomposition)}
\begin{enumerate}
\item[{\em ($A_n$)}] Given any definable $X_1 , \dots ,X_m \subseteq
I(R)^{n}$, there is a good decomposition of $I^n$ partitioning
each set $\st{X_i}$.

\item[{\em ($B_{n}$)}] If $f : X \rightarrow I(R)$, with $X \subseteq
I(R)^n$, is definable, then there is a good decomposition
$\mathcal{D}$ of $I^n$ such that for every open $C \in
\mathcal{D}$, either the set $\st(\Gamma f) \cap (C \times \mathbb{R})$ is
empty, or $f$ induces a function $g: C \rightarrow I$.
\end{enumerate}
\end{theorem}

\begin{proof}
We proceed by induction on $n$. Item $(A_1)$ holds by Lemma~\ref{stan2}.
We now assume $(A_n)$, $n>0$, and first prove $(B_n)$, and then $(A_{n+1})$.

Let $f: X \to I(R)$ be definable with $X\subseteq I(R)^n$.
Take a decomposition $\mathcal{P}$
of $R^n$ that partitions $I(R)^n$ and $X$ such that if $P$ is an
open cell of $\mathcal{P}$ contained in $X$, then
$f$ is continuously differentiable on $P$ and $\partial f/\partial x_i$ has
constant sign on $P$ for $i=1,\dots,n$.
Let $P\in \mathcal{P}$ be an open cell contained in $X$, and let
 $i\in \{1,\dots,n\}$.

Consider first the case that
$(\partial f/\partial x_i)>0$ on $P$, and put
$$P(i):= \{a\in P: (\partial f/\partial x_i)(a)> \Q\}, $$
so $\st P(i)\in \text{St}_n$ by Lemma~\ref{stan1}. Then
the set $\st P(i)\subseteq I^n$ has empty interior:
otherwise, Lemma~\ref{Qdensitylemma} gives a $\Q$-box $B\subseteq P(i)$,
but then $f$ could not be $\Q$-bounded on $B$, a contradiction.
In case $(\partial f/\partial x_i)\le 0$ on $P$, put
$$P(i):= \{a\in P: (\partial f/\partial x_i)(a) < \Q\}, $$
and then $\st P(i)\in \text{St}_n$ and
$\st P(i)$ has empty interior, by similar reasoning.

By $(A_n)$ we have a good decomposition $\mathcal{D}$ of $I^n$
partitioning $\st P$ and $\st \partial P$ whenever $P\in \mathcal{P}$ is
open and $P\subseteq I(R)^n$, and all $\st P(i)$, $1\le i \le n$, for which
$P\in \mathcal{P}$ is open and contained in $X$. We are going to show that
$\mathcal{D}$ has the property required by $(B_n)$.
Suppose $C\in \mathcal{D}$ is open.
Take $P\in \mathcal{P}$ such that $C\subseteq \st P$. Then $P$ is an open cell
contained in $I(R)^n$, so $C\cap \st \partial P = \emptyset$.

\medskip\noindent
{\em Claim 1}.\  $C^{\text{h}} \subseteq P$.

\medskip\noindent
To see this, let $a\in C^{\text{h}}$ and suppose $a\notin P$.
Take $b\in P$ with $\st a = \st b$, and note that the straight line segment
connecting $a$ to $b$ must contain a point $p\in \partial P$, but then
$\st p = \st a\in C$, a contradiction.

\medskip\noindent
Suppose now that $\st(\Gamma f)\cap (C\times \R)\ne \emptyset$. It
remains to show that then $f$ induces a function $C\to I$. It
follows from Claim 1 that $P\subseteq X$. Let $x\in C$ be given.
Then there is $y$ in $I$ with $(x,y)\in \st(\Gamma f)$, and there is
only one such $y$: if $(x,y_1), (x,y_2)\in \st( \Gamma f)$, with
$y_1\ne y_2$, take $a,b\in P$ with $\st a = x= \st b$ and $\st f(a)=
y_1$ and $\st f(b)= y_2$. By Claim 1, the infinitesimal line segment
connecting $a$ and $b$ is entirely contained in $P$, and by the Mean
Value Theorem this line segment must contain a point $p\in P(i)$
with $i\in \{1,\dots,n\}$, so $\st p = x \in \st P(i)$,
contradicting $C\cap \st P(i)=\emptyset$. Thus $f$ induces a
function $C\to I$. This finishes the proof of $(B_{n})$.

\bigskip\noindent
Towards proving $(A_{n+1})$, we first establish the following.

\medskip\noindent
{\em Claim 2}.\ Let $C_1,\dots,C_m\subseteq I^{n+1}$ be good cells;
then there is a good
decomposition of $I^{n+1}$ that partitions each $C_k$.

\medskip\noindent
To prove this, take functions $\phi_1,\dots,\phi_M$, ($M\in \N$),
where each $\phi_i\in G(D_i)$,
$D_i$ a good cell in $I^n$, such that each $C_k$ is of the form
$\Gamma \phi_i$ or $(\phi_i, \phi_j)$ (where in the latter case $D_i=D_j$).
Let $1\le i < j \le M$, and put
$$  D_{ij}:= \pi_n^{n+1}(\Gamma \phi_i \cap \Gamma \phi_j).$$
We show there are definable $P,Q\subseteq I(R)^n$
such that $D_{ij}=\st(P)\setminus \st(Q)$. To get such $P,Q$, take
$f: X \to I(R)$ and $g: Y \to I(R)$ with $X,Y\subseteq I(R)^n$,
such that $f$ induces $\phi_i$ and $g$ induces $\phi_j$.
It is easy to check that then
$$D_{ij}=\pi_n^{n+1}\big(\st(\Gamma f) \cap \st(\Gamma g)\big) \cap D_i \cap D_j,$$
so $D_{ij}$ has the desired form, by part (1) of Lemma~\ref{stan0} and
by Lemmas~\ref{stan1} and
\ref{gooddifference}.
By $(A_n)$ we can take a good decomposition
$\mathcal{D}$ of $I^n$ that partitions all $D_i$ and all $D_{ij}$.
It follows easily that there is a good decomposition $\mathcal{C}$ of
$I^{n+1}$ that partitions all $C_k$ such that
$\{\pi_n^{n+1}(C): C\in \mathcal{C}\}=\mathcal{D}$. This finishes the proof
of Claim 2.

\medskip\noindent To prove $(A_{n+1})$ we note that by cell decomposition in
the structure $R$ and Claim 2 it suffices to establish the
following special case:

\medskip\noindent
{\em Claim 3}.\ Let $X \subseteq I(R)^{n+1}$ be a cell in $R^{n+1}$; then
$\st(X)$ is a finite union of good cells in $\R^{n+1}$.

\medskip\noindent
Assume first that $X=\Gamma f$, with $f: p^{n+1}_{n} X \rightarrow
I(R)$. By $(A_{n})$ and $(B_{n})$, we have a finite partition
$\mathcal{P}$ of $\st(p^{n+1}_{n}X)$ into good cells, such that if
$C \in \mathcal{P}$ is open, then $f$ induces a function $C
\rightarrow I$, so $\st(X)\cap(C\times I)$ is a good cell. Consider
next a cell $C \in \mathcal{P}$ that is not open. Let
$i=(i_1,\dots,i_n)$ be such that $C$ is a good $i$-cell, take $k
\in\{1,\dots ,n \}$ such that $i_k = 0$, and consider the map
$$\pi :\ \mathbb{R}^{n+1} \rightarrow \mathbb{R}^{n}, \qquad
\pi(x_1,\dots ,x_{n+1})= (x_1 , \dots , x_{k-1} , x_{k+1} , \dots
,x_{n+1} ).$$
It is easy to see that $\pi |_{C \times I}: C \times
I \rightarrow \pi (C \times I)$ is a homeomorphism.
By Lemma \ref{proj}, the set
$\pi\big( \st(X) \cap (C \times I)\big)$ is a difference
of sets in $\text{St}_n$. Thus by $(A_{n})$,
$$\pi\big( \st(X) \cap (C \times I)\big)=  \bigcup_{i=1}^m E_i,$$
where $E_1,\dots,E_m\subseteq I^n$ are good cells. Then
$$\st(X) \cap (C \times I)\ = \bigcup_{i=1}^m \pi^{-1}(E_i)\cap (C \times I),$$and each  $\pi^{-1}(E_i)\cap (C \times I)$ is a good cell by
Lemma~\ref{inverse}.
It follows that Claim 3 holds for $X=\Gamma{f}$.

Next, assume that $X = (f , g)$ where $f,g:\ p^{n+1}_{n} X\  \to
I(R)$, $f<g$. By $(B_n )$, we have a finite partition $\mathcal{P}$
of $\st{(p^{n+1}_{n}X)}$ into good cells such that if $C \in
\mathcal{P}$ is open, then both $f$ and $g$ induce functions on $C$.
By $(A_n )$, we can take a finite partition $\mathcal{P'}$ of
$\st{(p^{n+1}_{n}X)}$ into good cells such that $\mathcal{P'}$
partitions each cell $C \in \mathcal{P}$ and for every open $C \in
\mathcal{P}$ it partitions the set $\{ \st{x} \in C: \; \st{f(x)}=
\st{g(x)}\}$. So if $C \in \mathcal{P'}$ is open, then $\st{X} \cap
(C \times I)$ is a good cell. If $C \in \mathcal{P'}$ is not open,
then we show
in the same way as in the case $X= \Gamma f$ that $\st{X} \cap (C
\times I)$ is a finite union of good cells.

\end{proof}

\noindent
A {\em good decomposition of $\R^n$\/} is a special kind of
partition of $\R^n$ into finitely many good cells. The definition is
by recursion on $n$:
\begin{enumerate}
\item[(i)] a good decomposition of $\R^1=\R$ is a collection $$\{(c_0 , c_1
),(c_2 , c_3 ), \dots ,(c_k , c_{k+1}), \{c_1 \} , \dots ,\{c_k \}
\}$$ of intervals and points in $\R$ where $ c_1 < \dots < c_k$ are
real numbers and $c_0 = - \infty$, $c_{k+1}= \infty$;
\item[(ii)] a good decomposition of $\R^{n+1}$ is a finite partition
$\mathcal{D}$ of
$\R^{n+1}$ into good cells such that
$\{\pi^{n+1}_{n} C:\ C\in \mathcal{D}\}$ is a good decomposition of $\R^n$.
\end{enumerate}
We set $J:= (-1,1) \subseteq \R$ and $J(R):= (-1,1) \subseteq R$. We
shall use the definable homeomorphism

$$\tau_n : R^n \rightarrow J(R)^{n}: (x_1 \dots ,x_n ) \mapsto
(\frac{x_1 }{\sqrt{1 + x_{1}^2}} , \dots , \frac{x_n
}{\sqrt{1+x_{n}^{2}}}),$$ and we also let $\tau_n$ denote the
homeomorphism

$$\tau_n : \mathbb{R}^n \rightarrow J^{n}: (x_1 \dots ,x_n ) \mapsto
(\frac{x_1 }{\sqrt{1 + x_{1}^2}} , \dots , \frac{x_n
}{\sqrt{1+x_{n}^{2}}}).$$ One easily checks that $\tau_1 : R
\rightarrow J(R)$ induces $\tau_1 : \mathbb{R} \rightarrow J$.

\begin{corollary}\label{reduction}
If $X_1 , \dots ,X_m \subseteq  R^n$ are definable, then there is a
good decomposition of $\mathbb{R}^n$ partitioning every $\st{X_i}$.
\end{corollary}

\begin{proof}
First note that by the remark right before this corollary, we have
$$\tau_n (\st{X_i })= \st{\tau_n (X_i )} \cap J^n$$ for all $i$.
Hence by Theorem \ref{thm} we have a good decomposition
$\mathcal{D}$ of $I^n$ partitioning $J^n$ and every $\tau_n (\st{X_i
})$.
\begin{trivlist}
\item{\em Claim.\/ }{If $D \subseteq J^n$ is a good cell, then
$\tau_{n}^{-1} (D) \subseteq \mathbb{R}^n$ is also a good cell.}
\item{We prove this by induction on $n$.
The claim clearly holds for $n=1$. Assume it holds for a certain $n\ge 1$,
and let $D \subseteq J^{n+1}$ be a good cell. Put $C:= \pi_n^{n+1} D$.
We first consider the case $D = \Gamma g$, where $g: C \rightarrow J$
is induced by $f: X \rightarrow R$ with $X \subseteq R^{n}$.
We can arrange that $X\subseteq J(R)^n$ and $f(X)\subseteq J(R)$. Then
$$\tau_{n+1}^{-1}(D)=\Gamma \tilde{g}, \quad \tilde{g}= \tau^{-1}
\circ g \circ \tau_{n}|_{\tau_{n}^{-1}C}:\
\tau_{n}^{-1} (C ) \rightarrow \R.$$
The set $\tau_{n}^{-1} C$ is a good cell by the inductive assumption and
$\tilde{g}$ is induced by $\tau^{-1} \circ f \circ
\tau_{n}|_{\tau_{n}^{-1}X}$. Thus $\tau_{n+1}^{-1}(D)=\Gamma \tilde{g}$
is a good cell in $\R^{n}$.

If $D = (-1 ,g)$ or $D=(g,1 )$, where $g$ is as above, then
$\tau_{n+1}^{-1}D = (-\infty , \tilde{g})$ or $\tau_{n+1}^{-1}D =
(\tilde{g},\infty )$, with $\tilde{g}$ defined as
above. We proceed likewise in the case $D = (g_1,g_2)$ with
$g_1, g_2: C \rightarrow J$.
Finally, if $D=C\times (-1,1)$, then we have
$\tau_{n+1}^{-1}D= (\tau_n^{-1}C) \times \R$.
  This concludes the proof of the claim.

}
\end{trivlist}
It follows that the collection of all $\tau^{-1}_{n} D$, where $D \in
\mathcal{D}$ and $D\subseteq J^n$, is a good decomposition of
$\mathbb{R}^{n}$ partitioning every $\st{X_i}$.

\end{proof}

\noindent
Theorem~\ref{intrthm} from the Introduction is now obtained as
follows. Let $n$ be given. By Corollary \ref{reduction} and
Lemma \ref{gooddifference}, the finite unions of sets
$\st(X) \setminus \st(Y)$ with definable $X,Y\subseteq R^n$ are exactly the
finite unions of good cells in $\R^n$, and these finite unions are also
the elements of a boolean algebra of subsets of $\R^n$. Also,
if $X$ is a finite union of good cells in $\R^n$, then
$X\times \R$ and $\R\times X$
are finite unions of good cells in $\R^{n+1}$. Finally,
the $\pi_n^{n+1}$-image of a finite union of good cells in $\R^{n+1}$
is clearly a finite union of good cells in $\R^n$.

\end{section}

\begin{section}{Closed sets and connected sets}

\medskip\noindent
In this section $n\ge 1$. For $x\in R^n$, and definable $Y\subseteq R^n$, put
$$|x|:= \max_i|x_i|, \quad d(x,Y):= \inf{\{|x-y|: \; y \in X\}} \in R\cup \{+\infty\}.$$
Likewise, for $x\in \R^n$ and any set $Y\subseteq \R^n$,
$$|x|:= \max_i|x_i|, \quad d(x,Y):= \inf{\{|x-y|: \; y \in X\}} \in \R\cup \{+\infty\}.$$

\begin{proposition}\label{closed}
The closed subsets of $\R^n$ definable in $\R_{\ind}$ are exactly
the sets $\st{X}$ with definable $X \subseteq R^n$.
\end{proposition}

\begin{proof}
The result will follow from Corollary \ref{reduction} once we show
that the closure of a good cell in $\R^{n}$ is of the form $\st{X}$
for some definable $X \subseteq R^n$. Let
$\epsilon$ range over positive infinitesimals. Let $C \subseteq \R^n$ be
a good cell.

\medskip\noindent
{\em Claim 1.\/ } There is an $r_0 \in \Q^{>0}$ and a
definable family $\{ X_r \}$ of subsets of
$R^n$, indexed by the
$r \in (0,r_0) \subseteq R$, such that
$$ 0 < r < r' < r_0\Longrightarrow X_{r'} \subseteq
X_{r};   \qquad \st{\big(\bigcap_{\epsilon} X_{\epsilon}\big)} = C .$$
The proof is by induction on $n$. If $C = \{ c \} \subseteq \R$,
then we take any
positive rational $r_0$ and $a \in R$ with $\st{a}=c$ and define $X_r
:= \{a \}$ for every $r \in (0,r_0 )$. If $C \subseteq \R$ is an
open bounded interval, then take $a,b \in R$ such that $\st{a}<
\st{b}$ are the endpoints of $C$ and let $X_r = (a+r , b-r)$ with
$r \in (0,r_0 )$ where $r_0$ is some positive rational $<
\frac{b-a}{2}$. The family $\{X_r \}$ has the desired properties.
The case that $C$ is an unbounded interval is left to the reader.

Assume the claim holds for certain $n\geq 1$, and let $D$ be a good
cell in $\R^{n+1}$. For $C := \pi_{n}^{n+1}D$, let $\{ X_r \}$ with
$r \in (0,r_0 )$ be a definable family as in the claim. We may
assume that $r_0 < 1$.

Consider first the case $D = \Gamma g$ where $g: C \rightarrow \R$
is induced by a definable $f: X \rightarrow R$, $X \subseteq R^{n}$.
After replacing $\{ X_r \}$ by $\{ X_r \cap X \}$ if necessary, we
may assume that $X_r \subseteq X$ for every $r$. We define
$$Y_{r} := \{ (x,y) \in R^{n+1} : \, x \in X_r \mbox{ and } f(x)=y \}.$$
It is easy to see that then $\st{\big(\bigcap Y_{\epsilon}\big)} = D$.

Next, assume $D=(\phi_1, \phi_2)$ with
$\phi_1, \phi_2: C \rightarrow \R$ induced by $f_1 : X_1
\rightarrow R$ and $f_2 : X_2 \rightarrow R$. Without loss of generality
 $X=X_1 = X_2$, $f_1 < f_2$ on $X$, and
$X_r \subseteq X$ for all $r \in (0,r_0 )$. For $r \in (0,r_0 )$, define
\[
\begin{array}{lll}
Y_{r} &:=& \{ (x,y) \in R^{n+1} : \, x \in X_r \mbox{ and }\\
&& f_1 (x) + \frac{f_2 (x)- f_1 (x)}{2} \, r < y < f_2 (x) -
\frac{f_2 (x)- f_1 (x)}{2}\, r \}.
\end{array}
\]
It is clear that if $0<r < r' <r_0$, then $Y_{r' } \subseteq
Y_{r }$. To get $D = \st{\big(\bigcap_{\epsilon} Y_{\epsilon}\big)}$, let
$x \in C^{h}$. Then $f_2 (x) - f_1 (x) > q$ for some $q \in
\Q^{>0}$, hence for $r\in (0,r_0)$ we have:
$\frac{f_2 (x)-f_1 (x)}{2} \, r$ is infinitesimal iff $r$ is infinitesimal.

The cases $D = C \times \R$, $D = ( - \infty , g)$, $D = (g
, \infty )$ are left to the reader.

\bigskip\noindent
{\em Claim 2.\/ } Let $\{ X_r \}$, $r \in (0,r_0 )$, be a
definable family as in Claim 1. Then there is an $\epsilon$ such that
$\st{X_{\epsilon }} = \closure{C}$.

\medskip\noindent
For each $\epsilon$ we have
$C \subseteq \st{X_{\epsilon}}$, hence $\closure{C} \subseteq
\st{X_{\epsilon }}$. Let $a\in \R^n\setminus \closure{C}$.
Pick $q_a\in \Q^{>0}$ with $d(a, \closure{C})>q_a$ and pick
$b_a\in \mathcal{O}^n$ with $\st(b_a)=a$.
Since $\st{X_r}\subseteq C$ for noninfinitesimal $r\in (0,r_0)$, this yields
$d(b_a,X_r) > q_a$  for such $r$. By o-minimality of $R$ this gives
$d(b_a, X_\epsilon)>q_a$ for all sufficiently large
(positive infinitesimal) $\epsilon$. Then by saturation we obtain an
$\epsilon$ such that $d(b_a, X_\epsilon)>q_a$ for {\em all\/}
$a\in \R^n\setminus \closure{C}$. For this $\epsilon$ we have
$a \notin \st{X_{\epsilon }}$ for all $a\in \R^n\setminus \closure{C}$,
and thus $\st{X_{\epsilon }} = \closure{C}$.
\end{proof}

\begin{lemma}
Suppose $X \subseteq \R^n$ is closed. Then $X^h$ is the intersection of a
type-definable set in $R^n$ with $\mathcal{O}^n$. In particular, if $X$ is
bounded, then $X^h$ is type-definable.
\end{lemma}
\begin{proof}
The complement of $X$ in $\R^n$ is a countable
union of open boxes, so $X= \bigcap_{i\in \mathbb{N}}\st Y_i$ where each
$Y_i\subseteq R^n$ is definable. Let
$$Y := \{ x \in R^n : \, d(x,Y_i ) < \frac{1}{n}\ \text{ for all $i$ and
all $n>0$}\}.$$
Then $Y$ is type-definable, and it is easy to check that
 $X^h = Y \cap \mathcal{O}^n$.
The second part of the lemma follows immediately from the first
part.
\end{proof}

\begin{proposition}
Let $X \subseteq R^n$ be definable, strongly bounded, and definably
connected. Then $\st{X} \subseteq \R^n$ is connected.
\end{proposition}

\begin{proof}
Assume towards a contradiction that $\st{X}$ is not connected. Then
$\st{X}$ is not definably connected with respect to the o-minimal structure
$\R_{\text{ind}}$, \cite{lou}, p. 59. So $\st{X} =
Y_1 \dot\cup Y_2$ where $Y_1 , Y_2$ are nonempty,
definable in $\R_{\text{ind}}$, and closed in $\st{X}$, and thus closed in
$\R^n$. Since
$$X = (X\cap Y_{1}^{h})\ \dot\cup\ (X \cap Y_{2}^{h}),$$ and $Y_{1}^{h}$, $Y_{2}^{h}$
are type-definable and disjoint, the sets $X \cap Y_{1}^{h}, X \cap
Y_{2}^{h}$ are definable, nonempty, and closed in $X$, a
contradiction.
\end{proof}

\end{section}

\begin{section}{Amenability}

Note that the proof of Corollary \ref{reduction} yields that if
$f: X \rightarrow R$ is definable with $X \subseteq R^n$, then there
is a good decomposition $\mathcal{D}$ of $\R^n$ such that if $D \in
\mathcal{D}$ is open, then either $\st{\Gamma f} \cap (D \times \R
)$ is empty or $f$ induces a function $D \rightarrow \R$.

\begin{lemma}\label{derivativeslemma}
Let both $C \subseteq \R^n$ and $X \subseteq R^n$ be open, and
suppose $f:X \rightarrow R$ is definable, $C^1$, and $f,
\frac{\partial f}{\partial x_1 }, \dots , \frac{\partial f}{\partial
x_n }$ induce functions $g, g_1 , \dots , g_n : C \rightarrow \R$.
Then $g$ is $C^1$ and $g_i = \frac{\partial g}{\partial x_i }$ for
all $i$.
\end{lemma}

\begin{proof}
Let $i \in \{1,\dots ,n \}$, let $e_i$ be the $i$-th unit vector in
$R^n$ or in $\R^n$ (according to the context), that is, $e_{ij} = 1$
if $i= j$ and $e_{ij}=0$ otherwise. Let $a \in C$, and take $b \in
C^h$ with $\st{b}=a$. Take $q \in  \Q^{>0}$ such that $a+te_i \in C$
for all $t\in \R$ with $|t| < q$, and also $b+te_i\in X$ for all
$t\in R$ with $|t|< q$. By the Mean Value Theorem we have, for $t\in
R$, $|t| < q$,
$$f(b+te_i)-f(b)= (\partial f/\partial x_i)(b+\theta e_i)\cdot t, \qquad
(\theta\in R,\ |\theta|\le |t|),$$ and taking standard parts this
gives for $t\in \R$, $|t| < q$,
$$g(a +te_i) - g(a)= g_i(a+ \tau e_i)\cdot t, \qquad
(\tau\in \R,\ |\tau|\le |t|).$$
Letting $t\in \R$ go to $0$ in this equality and using the continuity of
$g_i$ shows that
$\frac{\partial g}{\partial x_i }(a)$ exists and equals $g_i(a)$.
Because this
holds for all $i$ we conclude that $g$ is $C^1$.
\end{proof}

\begin{corollary}\label{derivatives}
Let $f:Y \rightarrow R$ with $Y \subseteq R^n$ be definable with
strongly bounded graph. Then there is a good decomposition $\mathcal{D}$
of $\R^n$ that partitions $\st{Y}$ such that if $D \in \mathcal{D}$ is
open and $D \subseteq
\st{Y}$, then $f$ is continuously differentiable on an open
definable $X \subseteq Y$ containing $D^h$, and $f, \frac{\partial
f}{\partial x_1 }, \dots , \frac{\partial f }{\partial x_n }$, as functions
on $X$, induce
functions $g, g_1 , \dots , g_n : D \rightarrow \R$ such that $g$ is
$C^1$ and $g_i = \frac{\partial g}{\partial x_i}$ for all $i$.
\end{corollary}

\begin{proof}
Since $\Gamma{f}$ is strongly bounded, we can reduce to the case that
$\Gamma f \subseteq I(R)^{n+1}$. Then the proof of $(B_n)$ in Theorem
\ref{thm} yields a good decomposition $\mathcal{D}$ of $I^{n}$ that
partitions $\st{Y}$
such that if $C \in \mathcal{D}$ is open and $C
\subseteq \st{Y}$, then there is an open $X \subseteq Y$ such that
$f|_{X}$ and $C$ satisfy the assumptions of Lemma
\ref{derivativeslemma}.
\end{proof}
The following notions are from \cite{nip}. Let $X,Y \subseteq R^n$
be definable. Define
\begin{align*} X \subseteq_0 Y: &\Longleftrightarrow\
\dim{(X\setminus Y)} < n,\\
   X=_0 Y: &\Longleftrightarrow\ X \subseteq _0
Y \text{ and }Y \subseteq_0 X.
\end{align*}
We say that a property holds for {\em
almost all\/} elements of $X$ if it
holds for all elements of $X$ outside a definable subset of
dimension $ < n$. We shall also use this notation and terminology when
$X,Y\subseteq \R^n$ are definable in  $\R_{\text{ind}}$,
with $\R_{\text{ind}}$ replacing $R$.

Let $V[n]$ be the additive monoid of all
definable $f:R^n \rightarrow R^{\geq 0}$ that are bounded with
bounded support, with addition being pointwise addition of
functions. If $f,g \in V[n]$, then by an {\em isomorphism\/} $\psi
:f \rightarrow g$ we mean a definable $C^1$-diffeomorphism $\psi : U
\rightarrow V$ with definable open $U,V \subseteq R^n$ such that
$\supp f \subseteq_0 U$, $\supp g \subseteq_0 V$, and
$$f(x)=|J\psi (x)| g(\psi (x))\ \text{ for almost all }x \in U,$$
where $|J\psi (x)|$ is the absolute value
of the determinant of the Jacobian matrix of $\psi$ at $x \in U$. We
call $f,g \in V[n]$ {\em isomorphic\/} if there is an isomorphism $f
\rightarrow g$. Note that $f\in V[n]$ is isomorphic to $0$ iff
$f(x)=0$ for almost all $x\in R^n$, and that
isomorphism is an equivalence relation on $V[n]$.

\begin{definition}
An {\em $n$-volume\/} is a finitely additive
$I: V[n] \rightarrow [0,\infty]$ such that $I(0)=0$ and $I$ is invariant under
isomorphisms.\footnote{Instead of isomorphism invariance, \cite{nip} requires
that $I(f)=I(g)$ if $f = \sum_{i=1}^{k}f_i$ and
$g = \sum_{i=1}^{k}g_i$ where $f_i ,g_i \in V[n]$ are isomorphic for
all $i$. This gives an equivalent definition.}
\end{definition}
Call a function $f \in V[n]$ {\em amenable\/} if there is an
$n$-volume $I$ such that $0 < I(f)<\infty$. Note that then $f$ is not
isomorphic to $0$.
Call $R$ {\em amenable for volumes\/}
if for every $n$, every $f \in V[n]$ not isomorphic
to $0$ is amenable.

\medskip\noindent
Question from \cite{nip}:
{\em is $R$ amenable for volumes?}
We give here a partial answer.

\bigskip\noindent
For $f\in V[n]$ we put $(0,f):= \{(x,y)\in R^{n+1}: 0 < y < f(x)\}$.
Let $SV[n]$ be the submonoid of $V[n]$ of all $f\in V[n]$ such that
$(0,f)$ is strongly bounded.

\begin{lemma}{\label{weakamenability}} There is a finitely additive
$I: SV[n] \rightarrow [0,\infty )$ with $I(0)=0$, such that
$I$ is invariant under isomorphism, and $I(f)>0$
for all $f \in SV[n]$ for which $\st{(0,f)}\subseteq \R^{n+1}$
has nonempty interior.
\end{lemma}

\begin{proof} Define $I: SV[n] \rightarrow [0,\infty )$ as
follows. Let $f \in SV[n]$, and take a good decomposition $\mathcal{D}$
of $\R^n$ such that  $f$ induces a function $f_D : D \rightarrow \R$
for every open $D \in \mathcal{D}$, and put $I(f) :=
\sum_{D}\int_{D} f_D$ where $\int$ is the Lebesgue integral and the
sum is taken over all open $D \in \mathcal{D}$. It is easy to see
that $I(f)$ is independent of the choice of such $\mathcal{D}$, and
that $0< I(f)< \infty $ if $\st{(0,f)}$
has nonempty interior in $\R^{n+1}$. It is also clear that $I$ is
finitely additive and
$I(0)=0$. Thus it is left to show that $I(f)=I(g)$
whenever $f,g \in SV[n]$ are isomorphic.

So let $f,g \in SV[n]$ be isomorphic.
Take a good decomposition $\mathcal{D}$ of $\R^n$ such that $f,g$
induce functions $f_D , g_D : D \rightarrow \R$ for every open $D
\in \mathcal{D}$. We define functions
$\hat{f}, \hat{g}: \R^n \to \R$ by
\begin{align*} \hat{f}(x)&= f_D(x) \text{ and }\hat{g}(x)=g_D(x) \text{ if $x\in D$ and $D\in \mathcal{D}$ is open},\\
  \hat{f}(x)&=\hat{g}(x)=0 \text{ if $x\notin D$ for all open
$D\in \mathcal{D}$}.
\end{align*}
Then $\hat{f}$ and $\hat{g}$ are $\R$-bounded with compact support
and definable in $\R_{\text{ind}}$. It is enough to show that
$\int \hat{f}=\int \hat{g}$.

Take a definable $C^1$-diffeomorphism
$\phi=(\phi_1,\dots,\phi_n) : U \rightarrow V$
where $U,V$ are open subsets of $R^n$ with $\supp{f} \subseteq_0
U$, $\supp{g} \subseteq_0 V$ and
$$f (x)= |J \phi (x)| \, g (\phi
(x)) \text{ for almost all }x \in U.$$ Note that $\phi{(\supp{f })}
=_0 \supp{g}$. So after replacing $\phi$ with
$\phi|_{\interior{\supp{f} \setminus Y}}$, where $Y$ is some
definable subset of $\supp{f}$ of dimension $<n$, we may assume that
the graph of $\phi$ is a strongly bounded subset of $R^{2n}$. Then,
applying Corollary \ref{derivatives} to the components of $\phi$ and
$\phi^{-1}$, we obtain open subsets $\hat{U} , \hat{V}$ of $\R^n$,
definable in $\R_{\text{ind}}$, such that each $\phi_i$ induces a
$C^1$-function $\hat{\phi}_i: \hat{U} \to \R$, and
$\hat{\phi}=(\hat{\phi}_1,\dots,\hat{\phi}_n) : \hat{U} \to \R^n$ is
a $C^1$-diffeomorphism onto its image $\hat{\phi}(\hat{U})=\hat{V}$,
each $\frac{\partial
\phi_i }{\partial x_j }$ induces $\frac{\partial \hat{\phi_i}}{\partial x_j }$
and $\supp{\hat{f} } \subseteq_0 \hat{U}$, $\supp{\hat{g }}
\subseteq_0 \hat{V}$. Then for almost all $x \in \hat{U}$ we have (taking
$y\in \mathcal{O}^n$ such that $\st{y}=x$),
$$\hat{f} (x) = \st{f (y)}=\st{|J\phi (y)|} \st{g (\phi (y))} =
|J\hat{\phi} (x)| \hat{g} (\hat{\phi} (x)),$$ hence
$\int \hat{f}=\int \hat{g}$.
\end{proof}

\bigskip\noindent
We let $B[n]$ be the collection of all bounded definable subsets of
$R^n$. Let $X,Y \in B[n]$. Then an {\em isomorphism\/} $\psi :
X \rightarrow Y $ is defined to be a definable $C^1$-diffeomorphism $\psi : U
\rightarrow V$, where $U,V \subseteq R^n$ are open and definable, $X
\subseteq_0 U$, $Y \subseteq_0 V$, $|J \psi (x)| = 1$ for almost all
$x \in X \cap U$, and $\psi (X \cap U) =_0 Y$. Note that $\psi$ is
an isomorphism $X \rightarrow Y$ iff it is an isomorphism $\chi^{}_X
\rightarrow \chi^{}_Y$. (Here $\chi^{}_X: R^n \to R$ is the characteristic
function of $X$.)
An {\em $n$-measure\/} is a finitely additive, isomorphism invariant
$\mu : B[n] \rightarrow [0,\infty ]$ with $\mu (\emptyset )=0$.

It is straightforward that an $(n+1)$-measure $\mu$ yields an
$n$-volume $I$ by $I(f):= \mu (0,f)$ for $f \in V[n]$, that an
$n$-volume $I$ gives an $n$-measure $\mu$ by putting $\mu
(X):=I(\chi^{}_X )$, and that $R$ being amenable for volumes is
equivalent to having for every $n$ and every $X \in B[n]$ with
nonempty interior an $n$-measure $\mu$ with $0<\mu (X)<\infty$.

Let $SB[n]$ be the collection of all strongly bounded definable
subsets of $R^n$. The proof of Lemma \ref{weakamenability} shows
that the finitely additive measure
$\mu=\mu^{(n)}: SB[n] \rightarrow [0, \infty )$ from \cite{bo}
is invariant under isomorphism; it also has the property that
$\mu (X)>0$ for all $X \in SB[n]$ such that
$\st{X}$ has nonempty interior.

\begin{theorem} There is for each $n$ an $n$-volume $I$ such that
$0<I(f) < \infty$ for all $f \in SV[n]$ for which $\st{(0,f)}$ has nonempty
interior in $ \R^{n+1}$.
\end{theorem}

\begin{proof}
By the above remarks it suffices to show that for all $n$ the
finitely additive $\mu=\mu^{(n)}:SB[n] \rightarrow [0,\infty )$
extends to an $n$-measure. We extend $\mu$ to $\mu^{\ast}: B[n]
\rightarrow [0,\infty ]$ as follows: if $X \in B[n]$ is isomorphic
to $Y \in SB[n]$, then $\mu^{\ast} (X) := \mu (Y)$; if $X \in B[n]$
is not isomorphic to any $Y\in SB[n]$, then $\mu^{\ast}(X):=
\infty$. Clearly, $\mu^{\ast} (\emptyset )=0$ and $\mu^{ \ast }$ is
invariant under isomorphism. We claim that $\mu^{\ast}$ is finitely
additive, and thus an $n$-measure. Let $X,Y \in B[n]$ be disjoint;
we need to show that then $\mu^{\ast}(X) + \mu^{\ast}(Y) =
\mu^{\ast}(X \cup Y)$. We can reduce to the case where $X \cup Y$ is
isomorphic to $Z$ where $Z \in SB[n]$; it remains to show that then
there are $X', Y' \in SB[n]$ isomorphic to $X$ and $Y$,
respectively. Let $\psi : U \rightarrow V$ be an isomorphism $X \cup
Y \to Z$; so $X \cup Y \subseteq_0 U$ and $Z \subseteq_0 V$. Then
$$ \psi |_{\interior{X \cap U}}:\  X\to X':=\psi (\interior{X \cap U})\cap Z$$ is an isomorphism and $X'\in SB[n]$, and likewise with $Y$.
\end{proof}

\end{section}

\begin{ack}
I would like to thank Lou van den Dries for many useful suggestions
concerning both the form and the contents of the current paper. It
was he who conjectured (while watching Letterman) Theorem 1.1.
Thanks also to Clifton F. Ealy for repeating a useful observation,
and to Kobi Peterzil for many helpful discussions and his
hospitality during a visit in Haifa. Lou van den Dries would like to
thank the author and Clifton F. Ealy for regularly playing pool with
him.
\end{ack}

\end{document}